\documentclass[11pt]{article}
\usepackage{xcolor}
\usepackage{amsthm}
\usepackage{amsfonts}
\usepackage{graphicx}
\usepackage{latexsym}
\usepackage{url}
\usepackage{bbm}
\usepackage{mathtools}
\mathtoolsset{showonlyrefs}

\newcommand{\R}{\mathbb{R}}

\newcommand{\E}{\mathbb{E}}
\newcommand{\N}{\mathbb{N}}

\renewcommand{\P}{\mathbb P}

\renewcommand{\S}{\mathbb S}

\newcommand{\al}{\alpha}
\newcommand{\be}{\beta}
\newcommand{\la}{\lambda}

\newcommand{\ga}{\gamma}
\newcommand{\ka}{\kappa}

\newcommand{\si}{\sigma}

\newcommand{\ep}{\varepsilon}

\newcommand{\de}{\delta}
\newcommand{\te}{\theta}

\newcommand{\De}{\Delta}
\newcommand{\om}{\omega}
\newcommand{\Om}{\Omega}
\newcommand{\ze}{\zeta}

\newcommand{\f}{\mathcal F}

\newcommand{\proba}{(\Omega ,\mathcal{F},(\f_t)_{t\geq0},\P)}

\newcommand{\toop}{\stackrel{\P}{\longrightarrow}}
\newcommand{\schw}{\stackrel{\raisebox{-1pt}{\textup{\tiny d}}}{\longrightarrow}}

\newcommand{\eqschw}{\stackrel{d}{=}}
\newcommand{\stab}{\stackrel{\mathcal{L}-s}{\longrightarrow}}

\newcommand{\toLt}{\stackrel{L^2}{\longrightarrow}}

\newcommand{\ucp}{\stackrel{\mbox{\tiny u.c.p.}}{\longrightarrow}}

\newcommand{\lan}{\langle}
\newcommand{\ran}{\rangle}
\newcommand{\lf}{\lfloor}
\newcommand{\rf}{\rfloor}

\newcommand{\bee}{\begin{equation}}
\newcommand{\eee}{\end{equation}}
\newcommand{\beea}{\begin{array}}
\newcommand{\eeea}{\end{array}}

\renewcommand{\theequation}{\arabic{section}.\arabic{equation}}

\theoremstyle{plain}
\newtheorem{prop}{Proposition}[section]
\newtheorem{cor}[prop]{Corollary}
\newtheorem{theo}[prop]{Theorem}
\newtheorem{lem}[prop]{Lemma}

\theoremstyle{definition}

\newtheorem{rem}[prop]{Remark}

\newcommand{\cadlag}{c\`adl\`ag}

\definecolor{darkblue}{rgb}{.1, 0.1,.8}
\definecolor{darkgreen}{rgb}{0,0.8,0.2}
\definecolor{darkred}{rgb}{.8, .1,.1}

\DeclareMathOperator{\im}{i}
\usepackage{dsfont}
\newcommand{\1}{\mathds{1}}

\begin{document}

\title{On estimation of quadratic variation for multivariate pure jump semimartingales \thanks{
The authors would like to thank Volodymyr Fomichov for helpful remarks.  
Mark Podolskij
gratefully acknowledges financial support of ERC Consolidator Grant 815703
``STAMFORD: Statistical Methods for High Dimensional Diffusions''. Johannes Heiny was supported by the project ``Ambit fields: Probabilistic properties and statistical inference'' funded by Villum Fonden and by the Deutsche Forschungsgemeinschaft (DFG) through RTG 2131 High-dimensional Phenomena in Probability – Fluctuations and Discontinuity.}} 
\author{Johannes Heiny \thanks{Department
of Mathematics, Ruhr University Bochum, 
E-mail: johannes.heiny@rub.de.} \and
Mark Podolskij\thanks{Department
of Mathematics and Department of Finance, University of Luxembourg,
E-mail: mark.podolskij@uni.lu.}}

\maketitle

\begin{abstract}
\noindent In this paper we present the asymptotic analysis of the realised quadratic variation for multivariate symmetric  $\be$-stable L\'evy processes,
$\be \in (0,2)$,  
and certain pure jump semimartingales. The main focus is on derivation of functional limit theorems for the  realised quadratic variation and its spectrum.
We will show that the limiting process is a matrix-valued $\be$-stable L\'evy process when the original process is symmetric  $\be$-stable, while
the limit is conditionally $\be$-stable in case of integrals with respect to   locally  $\be$-stable motions. These asymptotic results are mostly related
to the work \cite{DJT13}, which investigates the univariate version of the problem. Furthermore, we will show the implications for estimation of eigenvalues and eigenvectors of the quadratic variation matrix, which is a useful result for the principle component analysis. Finally, we propose a consistent subsampling procedure in the L\'evy setting to obtain confidence regions.

\ \

\bigskip


\end{abstract}

\section{Introduction} \label{sec1}
\setcounter{equation}{0}
\renewcommand{\theequation}{\thesection.\arabic{equation}}

During the past two decades statistics and limit theorems for It\^o semimartingales have received a great deal of attention in the literature. This has been mainly motivated by numerous applications in finance among other fields of science. A tremendous progress has been achieved in the statistical and probabilistic analysis of fine structure of  It\^o semimartingales in the infill asymptotic regime. A detailed account of the asymptotic theory for high frequency observations of semimartingales can be found in the monograph  \cite{JP12}.  

Statistical estimation of quadratic variation of It\^o semimartingales is probably one of the most important statistical problems in financial applications. Indeed, quadratic variation determines the variability of price processes in finance and it plays a crucial role in option pricing, predictions and other applications. During the past twenty years there has been a number of studies devoted to statistical inference for quadratic variation in various settings. Some accounts on this topic can be found in \cite{BGJPS06, BS04,J08,M01} among many other contributions. 
For a $d$-dimensional It\^o semimartingale $(Y_t)_{t\geq 0}$ defined on a filtered probability space $\proba$ and observed at time points $i\De_n$, $i=0,\ldots, \lf t/\De_n \rf$, the classical estimator of the quadratic 
variation $[Y]_t$ is given by the ``sum of squares'' 
\[
[Y]_t^n:= \sum_{i=1}^{\lf t/\De_n \rf} \left( Y_{i\De_n} - Y_{(i-1)\De_n} \right) \left( Y_{i\De_n} - Y_{(i-1)\De_n} \right)^{\top} \toop   [Y]_t,
\]
which we call the \textit{realised quadratic variation}. In the continuous It\^o semimartingale framework, where $(Y_t)_{t\geq 0}$ 
has the form
\[
Y_t=Y_0 + \int_0^t b_s ds + \int_0^t \si_s dW_s
\]
with $b$ being a $d$-dimensional drift, $\si$ a $\R^{d\times d}$-valued volatility process and $W$ a $d$-dimensional Brownian motion, 
estimation of quadratic variation has been investigated in e.g. \cite{BGJPS06, BS04}. In particular, under minimal assumptions on $b$ and $\si$,
they proved the following functional stable weak limit theorem. We refer to the next section for the formal definition of stable convergence.

\begin{theo} \label{Th1} (\cite[Theorem 5.4.2]{JP12})
Let $b$ be a predictable locally bounded drift and $\si$ a \cadlag~volatility process. Then we obtain the functional stable convergence
\[
\De_n^{-1/2} \left([Y]_t^n -  [Y]_t \right) \stab M_t,
\]
where, conditionally on $\f$, $(M_t)_{t\geq 0}$ is a Gaussian martingale with mean zero and conditional covariance function
$$\E[M_t^{jk}M_t^{j'k'}|\f ]= c^{jj'}_tc^{kk'}_t + c^{jk'}_tc^{kj'}_t, \qquad  c_t: = \si_t \si_t^{\top}. $$ 
\end{theo}

\noindent
The first result on this problem for general It\^o semimartingales with non-vanishing continuous martingale part appeared in \cite{J08}. 
The author investigated processes of the type
 \[
Y_t=Y_0 + \int_0^t b_s ds + \int_0^t \si_s dW_s +J_t,
\]
where $J$ denotes the jump part of $Y$. He proved the following result.

\begin{theo} (\cite[Theorem 5.4.2]{JP12})
Assume that $b$ is predictable and locally bounded drift, $\si$ is  \cadlag~and $J$ is a pure jump It\^o semimartingale with a locally bounded characteristic. Then we obtain the functional stable convergence
\[
\De_n^{-1/2} \left([Y]_t^n -  [Y]_{\De_n \lf t/\De_n \rf} \right) \stab M_t + \sum_{m:~T_m \leq t} \left( \De Y_{T_m} R_m^{\top} +
 R_m \De Y_{T_m}^{\top}  \right),
\]
where $(M_t)_{t \geq 0}$ has been introduced in Theorem \ref{Th1}, $(T_m)_{m\geq 1}$ (resp. $\De Y_{T_m}$) denote the jump times (resp. jump sizes) of $Y$, $R_m=\sqrt{\ka_m} \si_{T_m-} \Psi_m^+
+ \sqrt{1-\ka_m} \si_{T_m} \Psi_m^-$ and 
\[
(\ka_m)_{m\geq 1} \text{ i.i.d. } \sim \mathcal{U}(0,1), \qquad  (\Psi_m^+)_{m\geq 1} ,  (\Psi_m^-)_{m\geq 1}  \text{ i.i.d. } \sim \mathcal N_d (0, I_d). 
\]
Furthermore, conditionally on $\f$, the processes $(M_t)_{t \geq 0}$, $(\ka_m)_{m\geq 1}$, $(\Psi_m^+)_{m\geq 1}$ and $(\Psi_m^-)_{m\geq 1}$
are mutually independent. 
\end{theo} 

\noindent
Quite surprisingly, central limit theorems for quadratic variation only require very weak assumptions on the model in both settings.  Things are different in the pure jump setting, which has been recently studied in  \cite{DJT13}. In the latter framework the authors consider univariate stochastic integral processes, where the driving motion is locally $\be$-stable with $\be \in (0,2)$.  They show a stable weak limit theorem for the realised quadratic variation with a $\be$-stable limit in most interesting cases. Further related literature includes the articles \cite{T15,T19}, which investigate the estimation of the jump activity index and non-parametric inference for the spectral measure in the bivariate setting, respectively.    

The aim of this paper is to provide a weak limit theory for the realised quadratic variation in the setting of multivariate symmetric $\be$-stable L\'evy processes and related stochastic integral models. We will show that the limiting process is a matrix-valued $\be$-stable L\'evy motion in the pure L\'evy case and we will determine its directional measure. In the setting of integral models driven by a multivariate locally $\be$-stable L\'evy process we will prove that the limit is conditionally $\be$-stable. Another important contribution of our paper is the asymptotic analysis of the associated eigenvalues and eigenvectors, which is of major importance for the principal component analysis. This part provides an extension 
of the classical work \cite{A63}, which gave a complete theory for the principal component analysis in the Gaussian i.i.d. setting.

The paper is organised as follows. Section \ref{sec2} presents the models, notations and the main theoretical results. Section \ref{sec3} is devoted to construction of asymptotic confidence regions via a subsampling approach in the L\'evy setting. Finally, the proofs of the main results are collected in  Section \ref{sec4}.

\section{The model, notation and main results} \label{sec2}
\setcounter{equation}{0}
\renewcommand{\theequation}{\thesection.\arabic{equation}}

\subsection{Notation} \label{sec2.1} 
In this subsection we briefly introduce the main notations used throughout the paper. For $a \in \mathbb{C}$ we write $|a|$ to denote the norm of $a$.
For a vector or a matrix $x$ the transpose of $x$ is denoted by $x^{\top}$. The notation $\|x\|$ (resp. $\lan x,y \ran$) stands for the Euclidean norm of $x\in \R^d$ (resp. the scalar product of $x,y\in \R^d$).  
 We associate to the Frobenius norm $\|\cdot\|_{\text{tr}}$ the scalar product
\[
\lan A_1, A_2 \ran_{\text{tr}} := \text{tr} (A_1^{\top} A_2), \qquad    A_1,A_2\in \R^{d_1 \times d_2},
\]
where $\text{tr}$ denotes the trace.  We denote by $\|A\|_{\text{op}}$ the operator norm of $A \in  \R^{d_1 \times d_2}$. 
For a symmetric matrix $A \in \R^{d \times d}$ we write $\la_{\max}(A)$, $\la_{\min}(A)$ for the largest and the smallest eigenvalue of $A$, respectively. Furthermore,  $A_1\otimes A_2$ stands for the Kronecker product of $A_1$ and $A_2^{\top}$, and $A^{\otimes 2}:=A \otimes A$. In particular, for column vectors $x,y \in \R^d$ we have $x\otimes y=x y^{\top}$. We also introduce the symmetric tensor
\[
x\odot y = x y^{\top} + y x^{\top}.
\] 
The set $\S_d$ denotes the Euclidean unit sphere in $\R^d$.   For a \cadlag~stochastic process $(Y_t)_{t \geq 0}$ we denote by $Y_{t-}$ the left limit of $Y$ at point $t$ and by $\De Y_t = Y_t - Y_{t-}$ the jump at $t$. Throughout this paper $\De_n$ is a sequence of positive numbers satisfying $\De_n \to 0$ and we write
\[
\De_i^n Y:= Y_{i\De_n} - Y_{(i-1)\De_n}.
\]
If not stated otherwise, the asymptotic relations are with respect to $n \to \infty$ throughout this paper.
For stochastic processes $Y^n$ and $Y$ we denote by $Y^n \ucp Y$ the uniform convergence in probability, that is 
$$\sup_{t\in [0,T]}|Y^n_t - Y_t| \toop 0 \qquad \text{for any } T>0.$$ 
 In the following we will often use the notion of {\em stable convergence}. We recall that a sequence of random variables 
$(Y_n)_{n\in \N}$ 
defined on $(\Omega, \mathcal F, \mathbb P)$ is said 
to converge stably  with limit $Y$ ($Y_n \stab Y$) defined on an extension
$(\overline \Omega, \overline{\mathcal F}, 
\overline{\mathbb P})$ of the original probability space $(\Omega, \mathcal F, \mathbb P)$, 
iff for any bounded, continuous function $g$ and any bounded $\mathcal{F}$-measurable random variable $Z$ it holds that
\bee \label{defstable}
\E[ g(Y_n) Z] \rightarrow \overline{\E}[ g(Y) Z], \quad \text{as } n \rightarrow \infty.
\eee  
If not mentioned otherwise, the stable convergence is understood in the sense of Skorokhod $J_1$-topology for stochastic processes defined on the interval $[0,T]$.  We refer to \cite{AE78,R63} for a detailed exposition of stable convergence. 

Finally, we will deal with $\R^d$-valued (or $\R^{d\times d}$-valued) L\'evy processes $(Y_t)_{t\geq 0}$ without a Gaussian part. They are characterised 
by the L\'evy triplet $(\ga,0,\nu)$, i.e.
\[
\E[\exp(i\lan u,Y_1\ran)] = \exp\left(i\lan\gamma,u \ran + \int_{\R^d} \left\{ \exp(i \lan x,u \ran) - 1 - i \lan x,u \ran \1_{\{\|x\|\leq 1\}} \right\} \nu(dx)\right)
\] 
for $u,\ga \in \R^d$ and a measure $\nu$ satisfying $\nu(\{0\})=0$, $\int_{\R^d} (1\wedge \|x\|^2) \nu (dx)< \infty$. When we consider  
$\R^{d\times d}$-valued L\'evy processes we use the scalar product $\lan \cdot, \cdot \ran_{\text{tr}}$ instead.

\subsection{The setting} \label{sec2.2} 
We consider a filtered probability space $\proba$, satisfying the usual conditions, on which we define a $d$-dimensional symmetric $\be$-stable L\'evy process  $(L_t)_{t \geq 0}$ with L\'evy triplet $(0,0,G)$. Here $G$ denotes the L\'evy measure of $L$, which admits the representation
\bee \label{repG}
G(dx) = \frac{1}{ \rho^{1+\be}} d\rho\, H(d\te), 
\eee  
where $x=(\rho,\te) \in \R_+ \times \S_d$ and $H$ denotes a symmetric finite measure on $\S_d$ (called the \textit{directional measure}). Our main focus is on the classical realised quadratic variation, which is defined as
\bee \label{qvn}
[Y]_t^n:= \sum_{i=1}^{\lf t/\De_n \rf} (\De_i^n Y)^{\otimes 2},   
\eee
for any semimartingale $(Y_t)_{t\geq 0}$.  When $Y$ is a pure jump semimartingale, which in particular applies to $Y=L$, it holds that
\[
[Y]_t^n \toop [Y]_t:= \sum_{s\in [0,t]} (\De Y_s)^{\otimes 2} \qquad \text{as } \De_n \to 0.   
\] 
The main aim of this paper is to study the weak limit theory associated to the above convergence and its consequences for estimation of the spectrum of $[Y]_t$. We start with the following proposition.

\begin{prop} \label{prop1}
Assume that $\text{\rm span}(\text{\rm supp}(H))=\R^{d}$. Then the eigenvalues $\la_1,\ldots,\la_d$ of $[L]_t$ are all distinct and strictly positive $\P$-almost surely.  
\end{prop}

\begin{proof}
The assumption $\text{\rm span}(\text{\rm supp}(H))=\R^{d}$ obviously implies that $\text{\rm supp}(H)$ contains linearly independent vectors 
$x_1,\ldots,x_d \in \R^d$. After 
 a change of basis we can assume without loss of generality that $x_j=e_j$, $j=1,\ldots, d$, where $(e_j)_{1\leq j \leq d}$ denotes the standard basis of $\R^d$.

We recall that for two probability measures $\nu_1, \nu_2$ it holds that $\text{\rm supp}(\nu_1*\nu_2)= 
\text{\rm supp}(\nu_1)+ \text{\rm supp}(\nu_2)$, where $\nu_1*\nu_2$ denotes the convolution of $\nu_1$ and $\nu_2$. From this observation 
we deduce that there exist further vectors $x_{d+1},\ldots, x_q \in \R^d$ such that  
 the law $\nu^t$ of $[L]_t$ satisfies 
\[
\text{\rm supp} \left(\nu^t \right) = \left\{\sum_{j=1}^{q} a_j x_j^{\otimes 2}:~ a_j\geq 0    \right\}.
\]
Here $q$ is a number satisfying $d\leq q \leq d(d+1)/2$ and symmetric matrices $(x_j^{\otimes 2})_{1\leq j\leq q}$ are linearly independent. 
In other words, $\text{\rm supp} (\nu^t )$ is a convex cone of dimension $q$. According to e.g. \cite[Theorem 27.10]{S99}, $\nu^t$ has a Lebesgue 
density on  $\text{\rm supp} (\nu^t )$. For a symmetric matrix $A\in \R^{d\times d}$ 
let us denote by $\la_1(A),\ldots, \la_d(A)$ the real eigenvalues of $A$. 
We observe that the sets
\begin{align}
S_1&= \left\{A=\sum_{j=1}^{q} a_j x_j^{\otimes 2}:~  a_j\geq 0,~ \text{det}(A)=0   \right\},  \\[1.5 ex]
S_2 &=  \left\{A=\sum_{j=1}^{q} a_j x_j^{\otimes 2}:~  a_j\geq 0,~ \la_{j_1}(A)=\la_{j_2}(A) \text{ for some } j_1 \not= j_2   \right\}
\end{align}
are closed with an empty interior in $\text{\rm supp} (\nu^t )$. Indeed, consider the  set $S_1$ and define $f(a_1,\ldots, a_q):= \sum_{j=1}^{q} a_j x_j^{\otimes 2}$. Recalling that $x_j=e_j$, $j=1,\ldots, d$, and assuming that $\text{det}(f(a_1,\ldots, a_q))=0$ for some $a_j\geq 0$, we deduce 
 \[
 \text{det}(f(a_1+\ep,\ldots,a_d+\ep,a_{d+1}, \ldots, a_q)) = \text{det} (f(a_1,\ldots, a_q) + \ep I_d)\not =0 
 \]
 for some small enough $\ep>0$. Hence, the  set $S_1$ has an empty interior and the proof works similarly for the  set $S_2$. Consequently, both sets have Lebesgue measure $0$ on $\text{\rm supp} (\nu^t )$, which implies the statement of the proposition. 
\end{proof}

\noindent
This result will play an important role if we want to transfer the limit theory for $[L]_t$ to the spectrum of $[L]_t$.

\subsection{Limit theorem in the L\'evy case}

In this section we present the asymptotic theory for the realised quadratic variation in the setting of
$\beta$-stable L\'evy processes.
We introduce the sequence 
\bee \label{den}
\de_n=(\De_n \log (1/ \De_n))^{-1/\be}\,, \qquad n \ge 1\,, 
\eee   
which will turn out to be the rate of convergence for the estimator $[L]_t^n$. We obtain the following result.

\begin{theo} \label{th1} 
For any $\be \in (0,2)$ we obtain the functional stable convergence 
\bee \label{Lstab}
U_t^n:= \de_n \left( [L]_t^n - [L]_{\De_n \lf t/\De_n \rf} \right) \stab U_t ,
\eee
where $(U_t)_{t \geq 0}$ is an $\R^{d\times d}$-valued L\'evy process with characteristic triplet $(0,0,\nu_U)$ and the L\'evy measure
 $\nu_U$ is given by  
\bee \label{nuUdef}
 \nu_U(B) = \frac{1}{2\be} \int_{\S_{d\times d}} \mu(dz) \int_{\R_+} \1_B(\rho z) \rho^{-1-\be} d\rho, 
\qquad B \in \mathcal B(\R^d\odot \R^d), 
\eee
and 
$$\mu(z) = \int_{\S_d^2} \1_z\left(\frac{\te_1\odot \te_2}{\|\te_1\odot \te_2\|_{\text{\rm tr}}} \right) \left(2(1+\lan \te_1,\te_2 \ran^2)\right)^{\be/2} H(d\te_1) H(d\te_2), \qquad z\in \mathcal{B}(\S_{d\times d}),$$
where $\S_{d\times d}$ denotes the unit sphere with respect to the Frobenius norm $\|\cdot\|_{\text{\rm tr}}$ and $B$ is bounded away from $0\in \R^{d\times d}$. 
%\nu_U(B) = \frac{1}{2\be} \int_{\S_{d}^{\odot 2}} \mu(dz) \int_{\R^+} \1_B(\rho z) \rho^{-1-\be} d\rho, \qquad B \in \mathcal B(\R^d\odot \R^d)
%$$\mu (z) = \int_{\S_{d} \times \S_{d}} \1_{z}(\te_1 \odot \te_2) H(d\te_1) H(d\te_2)$$ 
Moreover, the process $U$ is defined on an 
extension $(\overline \Omega, \overline{\mathcal F}, (\overline{\mathcal F}_t)_{t\geq 0},
\overline{\mathbb P})$ of the original space $\proba$ and is independent of the $\si$-algebra $\f$. 
\end{theo}

\noindent
Let us give some remarks about the L\'evy measure $\nu_U$ of the limiting process $(U_t)_{t \geq 0}$. Since $\nu_U$ admits the representation 
\eqref{nuUdef}, where $\mu$ is a finite measure on the unit sphere $\S_{d\times d}$,  the process $(U_t)_{t \geq 0}$ is an $\R^{d\times d}$-valued $\be$-stable  L\'evy process. Furthermore, the support of the directional measure $\mu$ is given as
\[
\text{supp}(\mu)=\left\{\frac{\te_1\odot \te_2}{\|\te_1\odot \te_2\|_{\text{\rm tr}}}\in \S_{d\times d}  :~ \te_1,\te_2\in \text{supp}(H) \right\},
\]
which in particular shows that the jumps of $U$ have at most rank $2$.  The latter fact is not surprising since the jumps of $[L]_t^n$ and 
$[L]_{\De_n \lf t/\De_n \rf}$ have rank $1$. Finally, we remark that the specific centring via  $[L]_{\De_n \lf t/\De_n \rf}$ is needed to prove the \textit{functional} stable convergence as it guarantees that both processes jump at the same time. If one is interested in \textit{pointwise} stable convergence it suffices to use the more natural centring $[L]_t$.

\begin{rem} \label{remsym} \rm 
(a) The symmetry of the directional measure $H$, and hence of the L\'evy process $L$, is assumed for simplicity of the representation of the limiting
process $U$. As it has been demonstrated in \cite{DJT13} in the univariate setting, $U$ may contain an additional drift term when $H$ is not symmetric. 
Furthermore, according to the theory of  \cite{DJT13} one can relax the assumptions on $L$ to allow for certain
\textit{locally} $\be$-stable L\'evy processes. \\ 
(b) There exists an alternative representation of the L\'evy measure $\nu_U$ with respect to a different ``directional'' measure. Indeed, we may write
\begin{align*}
\nu_U(B) &= \frac{1}{2\be} \int_{\S_{d} \odot  \S_d} \mu'(dz) \int_{\R^+} \1_B(\rho z) \rho^{-1-\be} d\rho, 
\qquad B \in \mathcal B(\R^d\odot \R^d), \\
\mu' (z) &= \int_{\S_{d} ^2} \1_{z}(\te_1 \odot \te_2) H(d\te_1) H(d\te_2), \qquad  z\in \mathcal B(\S_d\odot \S_d).
\end{align*} 
In some sense it is a more natural representation of $\nu_U$. \qed
\end{rem}

\noindent
Now, we would like to determine the asymptotic theory for the eigenvalues and eigenvectors of the estimator $[L]_t^n$.  Let us recall a standard result on differentiability of eigenvalues/eigenvectors considered as functions of the underlying matrix. Consider a symmetric matrix $A_0$ with distinct eigenvalues $\la_1,\ldots,\la_d>0$ and let $v_1,\ldots,v_d\in \R^d$ be the corresponding eigenvectors with unit length. 
Then, using differential notation, we obtain the identities
\bee \label{formula}
d\la_i = v_i^{\top} (dA) v_i, \qquad dv_i=(\la_i I_d - A_0)^{+} (dA) v_i, 
\eee 
where $(\la_i I_d - A_0)^{+}$ denotes the \textit{Moore-Penrose pseudoinverse} of the matrix $\la_i I_d - A_0$. We recall that a Moore-Penrose pseudoinverse $A^+$ of a matrix $A \in \R^{d\times d}$ is defined via four properties: (i) $AA^+ A=A$, (ii) $A^+A A^+=A^+$, (iii) $AA^+$ is Hermitian and
(iv) $A^+ A$ is Hermitian. As a consequence of these identities we deduce the following statement.

\begin{cor} \label{cor1} 
Assume that $\text{\rm span}(\text{\rm supp}(H))=\R^{d}$. Denote by $\la(t)=(\la_1(t),\ldots, \la_d(t))^{\top}$ (resp. $\la^n(t)=(\la_1^n(t),\ldots, \la_d^n(t))^{\top}$) 
and $v(t)=(v_1(t),\ldots, v_d(t))$ (resp. $v^n(t)=(v_1^n(t),\ldots, v_d^n(t))$) the eigenvalues/eigenvectors of $[L]_t$ (resp. of $[L]_t^n$). 
Then, for a fixed $t>0$, we obtain the stable convergence
\bee \label{eigenstab}
\de_n\left(\la^n(t) -\la(t), v^n(t) -v(t) \right) \stab \left(\left(v_i(t)^{\top} U_t v_i(t) \right)_{1\leq i \leq d},  
\left((\la_i(t) I_d - [L]_t)^{+} U_t v_i(t)\right)_{1\leq i \leq d}    \right),
\eee
where the process $(U_t)_{t\geq 0}$ has been defined in Theorem \ref{th1}. 
\end{cor} 

\begin{proof} 
According to Proposition \ref{prop1} the random eigenvalues $(\la_i)_{1\leq i \leq d}$ are distinct and strictly positive $\P$-almost surely. Hence, the mappings $\la(t)$ and $v(t)$ seen as functions in $[L]_t$ are infinitely often differentiable and the first derivatives are given by \eqref{formula}. Thus, 
the result follows from the $\de$-method for stable convergence and Theorem \ref{th1}.  
\end{proof}

\begin{rem} \rm
In this remark  we discuss two standard examples of L\'evy processes $L$. \\
(a) (i.i.d. case): We consider a symmetric $\be$-stable L\'evy process $L=(L^1,\ldots, L^d)$ with i.i.d. components. In this setting the directional measure $H$ 
of $G$ is proportional to the uniform measure on the set  $(\pm e_i )_{1\leq i\leq d}$, 
where $( e_i )_{1\leq i\leq d}$ is the standard orthonormal basis of $\R^d$, i.e. 
\[
H(\{\pm e_i\}) =a>0, \qquad \forall 1\leq i \leq d.
\]
Since the components of $L$ have no common jumps we have $[L]_t= \text{diag}([L^1]_t, \ldots, [L^d]_t)$.  
According to Remark  \ref{remsym} the support of $\mu'$ is  
given by the set  $(\pm e_i \odot e_j )_{1\leq i\leq j\leq d}$ and it holds that 
\[
\mu'(\{\pm e_i \odot e_j\}) =a^2, \qquad 1\leq i\leq j\leq d.
\] 
In other words, the elements $(U_t^{ij})_{1\leq i\leq j\leq d}$
are independent one-dimensional symmetric $\be$-stable L\'evy processes. Finally,  by Corollary \ref{cor1}, the limit of $\de_n(\la^n(t) -\la(t))$   
is an $\R^d$-valued symmetric $\be$-stable process with i.i.d. components.\\
(b) (uniform case): Here we consider a symmetric $\be$-stable L\'evy process $L$ with directional measure
\[
H(d\te) = a \1_{\S_d}(\te) d\te, \qquad a>0. 
\]
In this setting $[L]_t$ is a matrix-valued $\be/2$-stable L\'evy motion with directional measure 
$\widetilde{H}(A)= a \int_{\S_d} \1_A(\te^{\otimes 2}) d\te$, $A\in  \mathcal B(\S_{d\times d})$. 
The measure $\mu'$ associated with the limiting process $U$ is given by
\[
\mu' (z) = a^2 \int_{\S_{d} ^2} \1_{z}(\te_1 \odot \te_2) d\te_1 d\te_2.
\]
In this case it seems hard to compute the exact asymptotic distribution of the eigenvalues $\de_n(\la^n(t) -\la(t))$, but according to Corollary \ref{cor1}
it is necessarily $\be$-stable conditionally on $\f$ with dependent components. 
\qed
\end{rem}

\subsection{An extension of Theorem \ref{th1} }

In this section we will extend the statement of Theorem \ref{th1} to more general pure jump It\^o semimartingales. We recall that a $d$-dimensional It\^o semimartingale $(Z_t)_{t\geq 0}$ is characterised by its (random) local  triplet $(b_t^{Z}, c^{Z}_t, F^{Z}_t)$, where $B_t^Z=\int_0^t
b_s^{Z} ds$ is the finite variation part of $Z$, $C_t^Z=\int_0^t c_s^{Z} ds \in \R^{d\times d}$ is the quadratic variation of the continuous part of $Z$, and $\nu(dt,dx)= dt \otimes F_t^Z(dx)$ is the compensator 
of the random measure associated with the jump part of $Z$.  Here the stochastic processes $b^V$
and $c^V$ are optional, and 
\[
\int_0^t \int_{\R^d} \left(\|x\|^2 \wedge 1 \right) F_s^Z(dx) <\infty \qquad \text{for all } t>0.
\] 
We now consider a pure jump It\^o semimartingale of the form
\begin{align} \label{xnew}
X_t = \int_0^t \si_{s-} ~d\widetilde{L}_s + Y_t,
\end{align}
where $\si$ is a \cadlag~$\R^{d\times d}$-valued volatility process. In this representation 
$\widetilde{L}$ is a locally $\be$-stable process and $Y$ is a pure jump It\^o semimartingale, which will not affect the limit theory of Theorem \ref{th1}. To give more precise definitions of all involved processes we switch to the Grigelionis decomposition of $\widetilde{L}$. We write
\[
\widetilde{L}_t = \int_0^t \int_{\R^d}  x \1_{\{\|x\|\leq 1\}} (q-\overline{q})(ds,dx) + 
\int_0^t \int_{\R^d}  x \1_{\{\|x\|> 1\}} q(ds,dx), 
\] 
where $q$ is an integer-valued random measure on $\R_+\times \R^d$ with compensator 
$\overline{q}(dt,dx)=dt\otimes v_t(dx)$,  and
\begin{align} \label{vtdef}
v_t(dx) = f_t G(dx) + v'_t(x) dx. 
\end{align}
Here $(f_t)_{t \geq 0}$ is a positive real-valued c\'adl\'ag process and $v'_t(x) $ is a predictable signed process such that 
\begin{align} \label{vprimecon}
\int_{\R^d} \left(\|x\|^{\beta^{\prime}} \wedge 1 \right) |v'_t(x)| dx \quad \text{is locally bounded},
\end{align}
with $\beta^{\prime}<\beta$ and $\beta^{\prime} \in (0,1)$. In other words, the semimartingale 
$\widetilde{L}$ has the local triplet $(0,0,v_t(dx))$ and the activity of small jumps are dominated
by the first term $f_t G(dx)$ in the decomposition \eqref{vtdef} since $\beta^{\prime}<\beta$. This explains the notion of a  locally $\be$-stable process. The process $Y$ is supposed to have a local triplet $(0,0,F^{Y}_t)$ such that 
\begin{align} \label{Ycond}
\int_{\R^d} \left(\|x\|^{\beta^{\prime}} \wedge 1 \right) F_t^{Y}(dx) \quad \text{is locally bounded}
,
\end{align}
where $\beta^{\prime}$ is the same exponent as appears in \eqref{vprimecon} (note that assuming 
that the exponents in \eqref{vprimecon} and  \eqref{Ycond} are the same is not a restriction; if this is not satisfied we can always take the maximum of the two).

We remark that the described setting of models \eqref{xnew} is quite similar to the general framework of 
\cite{DJT13,T19}. To better understand the fine structure of the model  we present an equivalent decomposition of the process $X$. Following the calculations of \cite[Section 6.1]{T19} (and recalling that $G$ is symmetric) we can write 
\begin{align} \label{alternX}
X_t = \int_0^t \widetilde{\si}_{s-} ~ dL^{\prime}_s + Y^{\prime}_t + Y_t,
\end{align}
where $L'$ has the same law as $L$, $ \widetilde{\si}_{s}=f_s^{1/\be} \si_s$ and 
\begin{align*}
Y^{\prime}_t &= \int_0^t \int_{\R^d}  x \1_{\{\|x\|\leq 1\}} (q_1-\overline{q}_1)(ds,dx) + 
\int_0^t \int_{\R^d}  x \1_{\{\|x\|> 1\}} q_1(ds,dx) \\
& - \int_0^t \int_{\R^d}  x \1_{\{\|x\|\leq 1\}} (q_2-\overline{q}_2)(ds,dx) -
\int_0^t \int_{\R^d}  x \1_{\{\|x\|> 1\}} q_2(ds,dx)
\end{align*}
where $q_j$ is an integer-valued random measure on $\R_+\times \R^d$ with compensator 
$\overline{q}_j(dt,dx)=dt\otimes v_t^{j}(x)dx$, $j=1,2$, and  
\[
v_t^{1}(x)=  |v'_t(x)|, \qquad  v_t^{2}(x)=  2|v'_t(x)| \1_{\{v'_t(x)<0\}}.
\]
The main result of Theorem \ref{th1} extends to the more general setting of processes $(X_t)_{t \geq 0}$  as follows.

\begin{prop} \label{prop2}
Assume that $(\widetilde{\si}_t)_{t\geq 0}$ is an It\^o semimartingale with local random characteristics $(b_t^{\widetilde{\si}}, c_t^{\widetilde{\si}}, F_t^{\widetilde{\si}})$ such 
that the processes $(b_t^{\widetilde{\si}})_{t\geq 0}$, $(c_t^{\widetilde{\si}})_{t\geq 0}$ and $\int (\|x\|^2_{\text{\rm tr}} \wedge 1) F_t^{\widetilde{\si}} (dx)$ are locally bounded. Furthermore, suppose that 
conditions  \eqref{vprimecon} and \eqref{Ycond} hold for $\beta^{\prime}<\be$ with $\beta^{\prime} \in (0,1)$. Then we obtain the stable convergence
\bee \label{Xstab}
 \de_n \left( [X]_t^n - [X]_{\De_n \lf t/\De_n \rf} \right) \stab \int_0^t \widetilde{\si}_{s-} ~dU_s~ \widetilde{\si}_{s-}^{\top},
\eee
where the process $U$ has been defined in Theorem \ref{th1}. 
\end{prop}

\begin{proof}
The proof of Proposition \ref{prop2} is obtained by a local approximation of the process $(\widetilde{\si}_t)_{t\geq 0}$, application of Theorem \ref{th1} and an argument showing the negligibility of the process
$Y^{\prime} +Y$. We first introduce the decomposition
\[
X_t = \widetilde{X}_t +  \widetilde{Y}_t \quad \text{with} \quad 
 \widetilde{X}_t =  \int_0^t \widetilde{\si}_{s-} ~ dL^{\prime}_s, \quad
  \widetilde{Y}_t = Y^{\prime}_t +Y_t.
\]
We start with the treatment of the process $\widetilde{X}$. 
Using It\^o formula we can write 
\[
\de_n \left( [\widetilde{X}]_t^n - [\widetilde{X}]_{\De_n \lf t/\De_n \rf} \right) = \de_n \sum_{i=1}^{\lf t/\De_n \rf} 
\int_{(i-1)\De_n}^{i\De_n} \left( \widetilde{X}_{s-} - \widetilde{X}_{(i-1)\De_n} \right)\odot  d\widetilde{X}_s.
\]
Applying the latter we obtain the decomposition 
\[
\de_n \left( [\widetilde{X}]_t^n - [\widetilde{X}]_{\De_n \lf t/\De_n \rf} \right) =  \sum_{i=1}^{\lf t/\De_n \rf} \left( \ze_i^n + \ze_i^{\prime n}  \right)
\]
with 
{\small
\begin{align*}
 \ze_i^n &= \de_n \widetilde{\si}_{(i-1)\De_n} \int_{(i-1)\De_n}^{i\De_n} \left( L^{\prime}_{s-} - L^{\prime}_{(i-1)\De_n} \right)\odot  dL^{\prime}_s  \widetilde{\si}_{(i-1)\De_n}^{\top}, \\[1.5 ex]
 \ze_i^{\prime n} &= \de_n \Big( \int_{(i-1)\De_n}^{i\De_n} \left( \widetilde{X}_{s-} - \widetilde{X}_{(i-1)\De_n} \right)\odot  d\widetilde{X}_s -  \widetilde{\si}_{(i-1)\De_n} \int_{(i-1)\De_n}^{i\De_n} \left( L^{\prime}_{s-} - L^{\prime}_{(i-1)\De_n} \right)\odot  dL^{\prime}_s  \widetilde{\si}_{(i-1)\De_n}^{\top}\Big) .
\end{align*}}
According to \cite[Lemma 6.9]{DJT13} and Theorem \ref{th1} we have the functional stable convergence
\[
 \sum_{i=1}^{\lf t/\De_n \rf}  \ze_i^n \stab \int_0^t  \widetilde{\si}_{s-} ~dU_s~  \widetilde{\si}_{s-}^{\top}.
\]
Thus, it suffices to show that $\sum_{i=1}^{\lf t/\De_n \rf}  \ze_i^{\prime n} \ucp 0$. Since the latter can be proved componentwise, we may use the univariate argument of  \cite[Proposition 7.2]{DJT13} to
conclude that 
\[
\de_n \left( [\widetilde{X}]_t^n - [\widetilde{X}]_{\De_n \lf t/\De_n \rf} \right) \stab 
 \int_0^t  \widetilde{\si}_{s-} ~dU_s~  \widetilde{\si}_{s-}^{\top}.
\]
Now, we show the negligibility of the contribution of $\widetilde{Y}$. First of all, we notice that due to assumptions  \eqref{vprimecon} and \eqref{Ycond} each component $\widetilde{Y}^j$ of $\widetilde{Y}=(\widetilde{Y}^1,\ldots, \widetilde{Y}^d)$ has the local triplet $(0,0, F_t^{\widetilde{Y}^j})$ and 
\begin{align} \label{Yjcond}
\int_{\R} \left(|y|^{\beta^{\prime}} \wedge 1 \right) F_t^{\widetilde{Y}^j}(dy) \quad \text{is locally bounded}
\end{align}
for $\beta^{\prime}<\be$ with $\beta^{\prime} \in (0,1)$. Let us set 
\[
R_t^n := \de_n \left( [X]_t^n - [X]_{\De_n \lf t/\De_n \rf} \right)
-\de_n \left( [\widetilde{X}]_t^n - [\widetilde{X}]_{\De_n \lf t/\De_n \rf} \right)
\]
and $R^n=(R^n(i,j))_{1\leq i,j\leq d}$. Since the diagonal terms $R^n(j,j)$ correspond to the univariate setting investigated in  \cite{DJT13}, we can use Proposition 7.1 therein under condition 
\eqref{Yjcond} to deduce that 
\[
R^n(j,j) \ucp 0. 
\]
Similarly, applying the identity $yz=((y+z)^2 - (y-z)^2)/4$, we can show that $R^n(i,j) \ucp 0$ 
for $i\not= j$. This completes the proof of Proposition \ref{prop2}. 
\end{proof}

\noindent
When the spectrum of the quadratic variation $[X]_t$ is non-degenerate in the sense of Proposition \ref{prop1}, 
a direct analogue of Corollary \ref{cor1} gives 
the asymptotic theory for estimation of eigenvalues/eigenvectors of $[X]_t$.  Unfortunately, in the more general setting of the model \eqref{xnew} it seems hard to find a simple condition that gives this non-degeneracy condition. A notable exception are processes $X$ with $Y=Y^{\prime}=0$ and $\widetilde{\si}$ being independent of $L^{\prime}$. In this framework it suffices to assume that the measure $G$ satisfies the condition of Proposition \ref{prop1} and $\text{rank}(\widetilde{\si}_{t-})=d$ $\P$-almost surely 
for all $t>0$. Indeed, under these conditions the proof of  Proposition \ref{prop1} remains valid conditionally on the process $(\widetilde{\si}_t)_{t\geq 0}$.

\section{A subsampling procedure} \label{sec3}
\setcounter{equation}{0}
\renewcommand{\theequation}{\thesection.\arabic{equation}}

We remark that the theoretical result of \eqref{Lstab} in Theorem \ref{th1} is hard to use for statistical applications since the directional measure $H$ 
on $\S_d$ is unknown in general. The estimation of this infinite dimensional object is far from obvious and even if $H$ were known there exist no reliable numerical methods to generate the limiting process $(U_t)_{t\geq 0}$. Instead our aim is to propose a subsampling method, which automatically adapts to the unknown limiting distribution.

%To fix ideas we consider a measurable function $f:\R^{d\times d} \to \R$, which is assumed to be continuously differentiable in the neighbourhood of 
%$[L]:=[L]_1$, and we are interested in constructing confidence regions for the random parameter $f([L])$ (in other words, we fix the time $t=1$).   For instance, we may consider the case $f(A)=\la_{\text{max}}(A)$, which is continuously differentiable at $[L]$ when the assumption of  Proposition \ref{prop1} is satisfied. Applying the $\de$-method for stable convergence we deduce
%that 
%\[
%\de_n \left( f([L]_t^n) - f([L])\right) \stab \lan \nabla f([L]) , U_1 \ran_{\text{tr}}.
%\]
The main idea is to construct $M$ independent copies $\ze_{n,m}$, $m=1,\ldots,M$, such that 
$$\P(\ze_{n,m}\in A) \to \overline \P\left(U_1 \in A \right ) \qquad \text{as } n \to \infty,$$ 
for any open cylindrical set $A \in \R^{d\times d}$ (recall that $\overline \P$ denotes the probability measure on the extended space). 
Then, as $M\to \infty$, we deduce that 
\[
\frac{1}{M} \sum_{m=1}^M \1_{\{\ze_{n,m} \in A\}}  \toLt  \overline \P\left(U_1 \in A \right ). 
\] 
Hence, the left hand side gives a consistent estimator of the unknown distribution of $U_1 $. Unfortunately, it seems impossible to achieve the goal with just one additional scale $M$, and we will use an additional parameter $k\to \infty$. We first divide the interval 
$(0,1]$ into $M$ equidistant blocks $I_m:=((m-1)/M, m/M]$ for $m=1,\ldots, M$. Then we introduce the empirical quadratic variation over the block
$I_m$ computed at frequency $\De_n$:
\bee \label{zmn}
z(\De_n)_m^n:= \sum_{i:~ (i-1)\De_n, i\De_n \in I_m} (\De_i^n L)^{\otimes 2}. 
\eee
Since the convergence rate $\de_n=:\de(\De_n)$ introduced in  \eqref{den} depends on the unknown parameter $\be \in (0,2)$, we need to construct 
its estimator. For $p\in (-1/2, 0)$ we define the statistic
\[
r_n:= \frac{\sum_{i=2}^{\lf 1/\De_n \rf} \|L_{i\De_n} - L_{(i-2)\De_n}\|^{p}}{\sum_{i=1}^{\lf 1/\De_n \rf} \|\De_i^n L\|^{p}} \toop 2^{p/\be},
\] 
where the convergence follows by self-similarity and the law of large numbers. Hence, $\widehat{\be}_n:= p \log(2)/\log(r_n)$ is a $\De_n^{-1/2}$-consistent estimator of $\be$ by the central limit theorem since $\E[\|\De_i^n L\|^{2p}]<\infty$ for any $p\in (-1/2, 0)$. Similarly, we can define an estimator $\widehat{\be}_{n,m}$, which is built upon observations on the interval $I_m$. By analogous reasoning this estimator satisfies 
$\widehat{\be}_{n,m}-\be=O_{\P}((\De_n M)^{-1/2})$. 
We now set 
\[
\widehat \de_m (\De_n) := (\De_n \log (1/ \De_n))^{-1/\widehat{\be}_{n,m}} \quad \text{and} \quad \ze_{n,m,k}:= \widehat \de_m (k\De_n) M^{1/\widehat{\be}_{n,m}}
\left(z(k\De_n)_m^n - z(\De_n)_m^n \right).
\]
Finally, we introduce the statistic 
\bee \label{Snx}
S_n^{M,k}(A):=  \frac{1}{M} \sum_{m=1}^M \1_{\{\ze_{n,m,k} \in A\}}. 
\eee
Let us briefly describe the intuition behind this statistic. First of all, we note that the terms $z(k\De_n)_m^n, z(\De_n)_m^n$, $m=1,\ldots,M$, are independent and identically distributed. In the next step we observe that, according to Theorem \ref{th1},  the quantity 
$\de_n(z(\De_n)_m^n - ([L]_{m/M} - [L]_{(m-1)/M}))$ is close in distribution to $U_{m/M} - U_{(m-1)/M}$, which due to self-similarity has the same distribution as $M^{-1/\be} U_1$. However, since the quadratic variation over $I_m$ is an unknown quantity, it needs to be replaced by an empirical quantity without affecting the asymptotic theory. To do so we replace the statistic $z(\De_n)_m^n$ by  $z(k\De_n)_m^n$ for some $k\to \infty$ while using   $z(\De_n)_m^n$ as a proxy for the quadratic variation over $I_m$. Finally, adjusting and estimating the convergence rate, we obtain the quantity 
$\ze_{n,m,k}$.  The formal convergence result is as follows.

\begin{prop} \label{propS}
Assume that $M,k\to \infty$ such that $kM\De_n \to 0$. For any open cylindrical set $A \in \R^{d\times d}$ it holds that 
\[
S_n^{M,k}(A) \toLt  \overline \P\left(U_1 \in A \right ).
\]
\end{prop}

\begin{proof}
Since $M\to \infty$ and the random variables $(\ze_{n,m,k})_{1\leq m \leq M}$ are i.i.d, we just need to prove that 
\[
\ze_{n,m,k} \schw U_1 \qquad \text{as }  M,k\to \infty, ~ kM\De_n \to 0,
\]
for a fixed $m$. We analyse the various errors associated with $\ze_{n,m,k}$. We first prove the convergence
\bee \label{conver1}
 \de (\De_n) M^{1/\be} \left( z(\De_n)_m^n - \left([L]_{\frac{m}{M}} - [L]_{\frac{m-1}{M}} \right) \right) \schw U_1 \qquad
 \text{as }  M\to \infty, ~ M\De_n \to 0.
\eee 
We set $RV(\De_n)=\sum_{i=1}^{\lf 1/\De_n \rf} (\De_i^n L)^{\otimes 2}$. Due to self-similarity of the L\'evy process $L$ and
 its quadratic variation $[L]$, we obtain the identity in distribution 
\begin{align*}
\de (\De_n) M^{1/\be} \left( z(\De_n)_m^n - \left([L]_{\frac{m}{M}} - [L]_{\frac{m-1}{M}} \right) \right) &\eqschw
\de (\De_n) M^{-1/\be} \left(RV(M\De_n) - [L]_1\right) + o_{\P}(1). 
\end{align*} 
We conclude from Theorem \ref{th1} that $\de (M\De_n)\left(RV(M\De_n) - [L]_1\right) \schw U_1$. On the other hand, we have
that $\de (\De_n) M^{-1/\be}/ \de (M\De_n) \to 1$. Hence, the convergence in \eqref{conver1} holds.
Applying this convergence to $k \De_n$ instead of $\De_n$ we also deduce the convergence
\[
 \de (k\De_n) M^{1/\be} \left( z(k\De_n)_m^n - \left([L]_{\frac{m}{M}} - [L]_{\frac{m-1}{M}} \right) \right) \schw U_1.
\]
Combining the latter with the original statement \eqref{conver1}, we obtain that 
\[
 \de (k\De_n) M^{1/\be} \left( z(k\De_n)_m^n - z(\De_n)_m^n\right) \schw U_1,
\]
because $k\to \infty$. Finally, recalling that $\widehat{\be}_{n,m}-\be=O_{\P}((\De_n M)^{-1/2})$, it holds that 
\[
\widehat \de_m (k\De_n) M^{1/\widehat\be_{n,m}}  - \de (k\De_n) M^{1/\be} =  O_{\P}\left(\de (k\De_n) \De_n^{1/2} M^{-1/2} \log(1/\De_n) \right). 
\] 
Consequently, we obtain the statement of  Proposition  \ref{propS}.
\end{proof}

\begin{rem} \rm
We remark that the ratio based estimator $\widehat{\beta}_n$  is just one potential proxy for the stability parameter $\beta$. Alternatively, we could use any other 
estimator of $\be$ with polynomial convergence rate in the subsampling procedure. \qed
\end{rem}

\noindent
Suppose that we are aiming to estimate a quantity $f([L]_1)$ for a function $f:\R^{d \times d} \to \R$, which is continuously differentiable in a neighbourhood of $[L]_1$. One important example is the function $f(A)=\la_{\max}(A)$, which satisfies the differentiability condition under assumptions of Proposition \ref{prop1}. 
 Applying the $\de$-method to Theorem \ref{th1} we deduce the stable convergence
\[
\de_n\left(f([L]_1^n) -f([L]_1) \right) \stab \lan \nabla f([L]_1), U_1 \ran_{\text{tr}}. 
\]
In this case we can apply a similar subsampling procedure to obtain confidence regions for the unknown parameter.  Indeed, defining the quantity
\[
\ze_{n,m,k}(f):= \widehat \de_m (k\De_n) M^{1/\widehat\be_{n,m}}
\left(f\left(z(k\De_n)_m^n\right) - f\left(z(\De_n\right)_m^n )\right),
\]
we deduce the convergence
\bee
S_n^{M,k}(f,w):=  \frac{1}{M} \sum_{m=1}^M \1_{\{\ze_{n,m,k}(f) \leq w\}} \toLt  \overline \P \left(\lan \nabla f([L]_1), U_1 \ran_{\text{tr}} \leq w \right) 
\eee
for any $w\in \R$.

In the framework of integral processes $X$ defined in  \eqref{xnew}, things seem to be more complicated (even when $Y^{\prime}+Y=0$). We can propose a modified subsampling method, but, due to the stochastic nature of $\widetilde{\si}$, it will only assess the \textit{conditional} distribution of the limit in  \eqref{Xstab} given the path $(\widetilde{\si}_s(\om))_{s\geq 0}$ for a fixed $\om \in \Om$. However, this does not seem to suffice to construct confidence regions for the quadratic variation $[X]_t$.

\section{Proof of Theorem \ref{th1} } \label{sec4}
\setcounter{equation}{0}
\renewcommand{\theequation}{\thesection.\arabic{equation}}

All positive constants appearing in the proofs are denoted by $C$ although they may change from line to line. 

The basic idea behind the proof of Theorem \ref{th1} is to show the functional weak convergence (with respect to the Skorokhod $J_1$-topology)  
\[
\left( L_{\De_n \lf t/\De_n \rf}, U_t^n \right) \schw \left( L_t, U_t \right),   
\]
which implies the functional stable convergence $U^n \stab U$ according to e.g. \cite[Lemma 6.9]{DJT13}. Following the arguments of 
\cite{DJT13} (cf.~the proof of Proposition 7.3 therein) the laws of $L_{\De_n \lf \cdot /\De_n \rf}$ and $U^n$ factorise asymptotically, which guarantees the independence of the limits $L$ and $U$. Hence, the crucial step is the proof of the functional weak convergence
\bee \label{Uconv}
U^n \schw U,
\eee
which we will show in the following. 

\subsection{Main decompositions}

We are mainly following the decompositions proposed in \cite{DJT13} adapted to the multivariate setting. First of all, instead of dealing with the original L\'evy measure
defined at \eqref{repG}, we may truncate it and  work with $G$ restricted to the unit ball $\|x\|\leq 1$ (see the argument behind 
\cite[Assumption S2]{DJT13}).  So, from now on we assume that 
\[
G(dx) = \frac{\1_{(0,1]}(\rho)}{ \rho^{1+\be}} d\rho H(d\te), \qquad x=(\rho,\te) \in \R_+ \times \S_d.
\]
An important decomposition is given by 
\[
L=M(v) + A(v), \qquad v \in (0,1),
\]
where $A(v)_t=\sum_{s\leq t} \De L_s \1_{\{\|\De L_s\|>v\}}$, which corresponds to the classical L\'evy-It\^o decomposition (recall that the directional measure $H$, and hence $G$, is symmetric). When $\be> 1$, $M(v)$ is a martingale.
% and for $\be<1$ we can write $M(v)_t=\sum_{s\leq t} \De L_s \1_{\{\|\De L_s\|\leq v\}}$.  
Now, we set
\[
v_n = 
\begin{cases}
\De_n^{1/(2\be)} \log(1/\De_n)\,, & \text{if } \be>1 \\
(\De_n \log(1/\De_n))^{1/\be}\,, & \text{if } \be \leq 1
\end{cases} 
\] 
and define $M^n=M(v_n), A^n=A(v_n)$.  We note that the process $M^n$ has the L\'evy triplet $(0,0,G(dx) \1_{\{\|x\| \leq v_n \}})$ and $A^n$ is a compound Poisson process with intensity $\overline{G}(v_n):= G(\{x\in \R^d:~v_n<\|x\| \leq 1\})$ and jump distribution $G(dx) \1_{\{v_n<\|x\| \leq 1 \}}/ \overline{G}(v_n)$, and $M^n$ and $A^n$ are independent.

Finally, we denote by $\tau(n,i)$ (resp. $T(n,i)_j$) the number of jumps (resp. the time of the $j$th jump) with norm larger than $v_n$
in the interval $((i-1)\De_n, i\De_n]$.  Due to It\^o formula we have that 
\bee \label{sumstat}
U_t^n= \de_n \left( [L]_t^n - [L]_{\De_n \lf t/\De_n \rf} \right) = \sum_{i=1}^{\lf t/\De_n \rf} \xi_i^n, 
\qquad \xi_i^n =  \de_n \int_{(i-1)\De_n}^{i\De_n} \left( L_{s-} - L_{(i-1)\De_n} \right)\odot  dL_s.
\eee
Now, we further decompose the quantity $\xi_i^n$ in terms of $M^n$ and $A^n$, and according to the number of jumps of $A^n$ within the interval 
$((i-1)\De_n, i\De_n]$. More specifically, we have that $\xi_i^n=\sum_{j=1}^5 \xi_i^n(j) $ with 
\begin{align*}
\xi_i^n(1) &= \de_n \De_i^n M^n \odot  \De L_{T(n,i)_1} \1_{\{\tau(n,i)=1\}} \\[1.5 ex]
\xi_i^n(2) &= \de_n \De L_{T(n,i)_1}  \odot  \De L_{T(n,i)_2} \1_{\{\tau(n,i)=2\}} \\[1.5 ex]
\xi_i^n(3) &= \de_n \De_i^n M^n \odot  \De_i^n A^n \1_{\{\tau(n,i)\geq 2\}} \\[1.5 ex]
\xi_i^n(4) &=  \de_n \int_{(i-1)\De_n}^{i\De_n} \left( M_{s-}^n - M_{(i-1)\De_n}^n \right)\odot  dM_s^n \\[1.5 ex]
\xi_i^n(5) &=  \de_n \int_{(i-1)\De_n}^{i\De_n} \left( A_{s-}^n - A_{(i-1)\De_n}^n \right)\odot  dA_s^n \1_{\{\tau(n,i)\geq 3\}}
\end{align*} 
We will see that $\xi_i^n(1)$ represents the dominating part when $\be>1$, while $\xi_i^n(2)$ is dominating when $\be \leq 1$.

\subsection{Preliminary results} 
In this subsection we demonstrate some technical results, which are necessary to prove Theorem \ref{th1}. We start with a number of conditions that ensure negligibility of certain partial sums. The result below is a direct multivariate extension of \cite[Lemma 6.6]{DJT13}.

\begin{lem} \label{lemneglig}
Let $\eta_i^n$ be  $\R^{d\times d}$-valued $\f_{i\De_n}$-measurable random variables. Then each of the following conditions implies 
the uniform convergence $\sum_{i=1}^{\lf t/\De_n \rf} \eta_i^n \ucp 0$:

\begin{align}
\label{cond1} &\sum_{i=1}^{\lf t/\De_n \rf}\E[\|\eta_i^n\|_{\text{\rm op}} \wedge 1] \to 0 , \\[1.5 ex]
\label{cond2} & \sum_{i=1}^{\lf t/\De_n \rf} \E[\eta_i^n|~\f_{(i-1)\De_n}] \ucp 0 \quad \text{and} \quad  \sum_{i=1}^{\lf t/\De_n \rf}\E[\|\eta_i^n\|_{\text{\rm op}}^2|~\f_{(i-1)\De_n}] 
\toop 0 , \\[1.5 ex]
\label{cond3} & \left.
\begin{array}{l}
\sum_{i=1}^{\lf t/\De_n \rf} \E[\eta_i^n \1_{\{ \|\eta_i^n\|_{\text{\rm op}}\leq 1  \}}|~\f_{(i-1)\De_n}] \ucp 0, \quad 
\sum_{i=1}^{\lf t/\De_n \rf} \P (\|\eta_i^n\|_{\text{\rm op}}>1 )\to 0, \\[4.5 ex]
\sum_{i=1}^{\lf t/\De_n \rf} \E[\|\eta_i^n\|_{\text{\rm op}}^2 \1_{\{ \|\eta_i^n\|_{\text{\rm op}}\leq 1  \}}] \to 0.
\end{array}
\right\}
\end{align}
The operator norm can be replaced by any other norm on $\R^{d\times d}$. 
\end{lem}

\noindent
In the next lemma we demonstrate some inequalities for the moments of  $M^n$ and related processes. They both follow
from \cite[Proposition 2.1.10]{JP12}. 

\begin{lem} \label{momentineq}
Let $W$ be a predictable $\R^d$-valued process and $u>0$ fixed. Then it holds that 
\begin{align} \label{mnineq1}
&\E\left[\sup_{s\leq t} \left \|\int_{u}^{u+s} W_s \odot dM^n_s \right\|^p_{\text{\rm tr}} \right] \leq C  v_n^{p-\be} \E\left[ \int_{u}^{u+t} \|W_s\|^p ds\right], 
\\[1.5 ex]
\label{mnineq2} &\E\left[ \sup_{s\leq t} \left \| M^n_{u+s} - M^n_{u} \right\|^p \right] \leq C  t v_n^{p-\be}, 
\end{align}
for  $1\leq \be <p\leq 2$ when $\be\geq 1$ and $\be<p\leq 1$ if $\be<1$. 
\end{lem}

\noindent
Below we state a number of inequalities related to the compound Poisson part of the L\'evy process $L$. 
They directly follow from the univariate inequalities of \cite[Lemma 6.2 and 6.3]{DJT13}, since the objects 
$\tau(n,i)$ and $\|\De L_{T(n,i)_k} \|$ are one-dimensional.

\begin{lem} \label{Poissonest}
Recall the definition of random variables $\tau(n,i)$ and $T(n,i)_j$ from the previous subsection and let $w>0$, 
$b:= H(\S_d)/\beta$. 
\begin{itemize}
\item[(i)] For any $1\leq j \leq m$, it holds on the set $\{\tau(n,i) \geq j-1 \}$ that
\bee \label{tauineq}
\P\left(\tau(n,i) \geq m \right) \leq C \left(\De_n / v_n^{\be}\right)^{m-j+1}.  
\eee
\item[(ii)] For any $1\leq j \leq m$, it holds on the set $\{\tau(n,i) \geq j-1 \}$ that
\bee \label{Zineq2}
\E\left[(\|\De L_{T(n,i)_k} \| \wedge w )^p \1_{\{ \tau(n,i)\geq m\}} \right] \leq C
\begin{cases}
\De_n \left(\frac{b\De_n}{v_n^{\be}}\right)^{m-j} w^{p-\be} & \text{for } p>\be \\
\De_n \left(\frac{b\De_n}{v_n^{\be}}\right)^{m-j} \log(1/\De_n) & \text{for } p=\be \\
\De_n \left(\frac{b\De_n}{v_n^{\be}}\right)^{m-j} v_n^{p-\be} & \text{for } p<\be.
\end{cases}
\eee
\item[(iii)] For any $1\leq j \leq k< r \leq m$, it holds on the set $\{\tau(n,i) \geq j-1 \}$ that
\begin{align} \label{Zineq3}
&\E\left[\left(\|\De L_{T(n,i)_j} \| \|\De L_{T(n,i)_k} \|  \wedge w \right)^p \1_{\{ \tau(n,i)\geq m\}} \right] \\[1.5 ex]
&\leq C
\begin{cases}
\De_n^2 \log(1/\De_n) \left(\frac{b\De_n}{v_n^{\be}}\right)^{m-j-1} w^{p-\be} & \text{for } p>\be \\[1.5 ex]
\De_n^2( \log(1/\De_n))^2 \left(\frac{b\De_n}{v_n^{\be}}\right)^{m-j-1} & \text{for } p=\be. 
\end{cases}
\end{align}
\end{itemize}
\end{lem}

\subsection{Proof of Theorem \ref{th1} in the case $\be \in (1,2)$}

\subsubsection{Negligible terms}

We have $\E[\xi_i^n(4)|~\f_{(i-1)\De_n}]=0$.
Using \eqref{mnineq1} and \eqref{mnineq2}  for $p=2$ we get
\[
\E[\|\xi_i^n(4)\|_{\text{tr}}^2] \leq C \De_n^2 \de_n^2 v_n^{4-2\be} = C\De_n (\log(1/\De_n))^{4-2(\be +1/\be)}.
\] 
Since $\be+1/\be>2$ for $\be>1$, we deduce  for $t>0$ that
$$ \sum_{i=1}^{\lf t/\De_n \rf} \E[\|\xi_i^n(4)\|_{\text{tr}}^2| \f_{(i-1)\De_n}] \toop 0\,.$$  
Thus,  by condition \eqref{cond2} we obtain that
\[
\sum_{i=1}^{\lf t/\De_n \rf} \xi_i^n(4) \ucp 0.
\]
For $j=2,3,5$ we have 
\[
\E[\|\xi_i^n(j)\|_{\text{op}} \wedge 1] \leq \P(\tau(n,i) \geq 2) \leq C\left(\De_n / v_n^{\be}\right)^{2} = C \De_n (\log(1/\De_n))^{-2\be}
\]
by  \eqref{tauineq} applied to $m=2$. Hence, we conclude that $\sum_{i=1}^{\lf t/\De_n \rf} \E[\|\xi_i^n(j)\|_{\text{op}} \wedge 1] \to 0$ and we deduce
by \eqref{cond1} that
\[
\sum_{i=1}^{\lf t/\De_n \rf} \xi_i^n(j) \ucp 0 \qquad \text{for } j=2,3,5.
\]  
Putting things together we have shown that $\sum_{i=1}^{\lf t/\De_n \rf} \xi_i^n(j) \ucp 0$ for $j=2,3,4,5$.

\subsubsection{The dominating term}

In this subsection we treat the term $\xi_i^n(1)$, which constitutes the dominating part in the case $\beta>1$. We will show the functional weak convergence 
$\sum_{i=1}^{\lf t/\De_n \rf} \xi_i^n(1) \schw U_t$, and hence \eqref{Uconv}, by analysing the characteristic function. Set $\im=\sqrt{-1}$. Defining 
\[
\varphi_t^n (u) = \E\left[\exp\left(\im \left \lan u,  \sum_{i=1}^{\lf t/\De_n \rf} \xi_i^n(1) \right \ran_{\text{tr}} \right)\right] \quad \text{and}
\quad \varphi_t (u)=  \E\left[\exp\left(\im \left \lan u,  U_t \right \ran_{\text{tr}} \right)\right], 
\] 
it suffices to prove that $\varphi^n (u) \ucp \varphi (u) $ for any  $u \in \R^{d\times d}$ (see \cite[Corollary VII.4.43]{JS03}; here there is no randomness involved and $\ucp$ just stands for uniform convergence in time  on compact intervals). We note
 that $(\xi_i^n(1))_{1 \leq i \leq \lf t/\De_n \rf}$ is a sequence of i.i.d.~random variables and introduce the normalised characteristic function of 
 $\xi_1^n(1)$:
\bee \label{Rn}
R^n (u) =\E\left[\exp(\im \lan u,  \xi_1^n(1) \ran_{\text{tr}})\right] -1, \qquad u \in \R^{d \times d}. 
\eee
We will use the following well known statement from analysis (see for example \cite[Lemma 6.7]{DJT13}). Let $a_i^n$ be  complex numbers. Then it holds 
\begin{align} \label{lemhelp}
\sum_{i=1}^{\lf t/\De_n \rf} a_i^n \ucp g(t) \quad \text{and} \quad g \text{ is continuous } 
 \Longrightarrow \prod _{i=1}^{\lf t/\De_n \rf} (1+ a_i^n ) \ucp \exp(g(t)). 
\end{align}
Applying the result of \eqref{lemhelp} to the setting $a_i^n=R^n (u)$, it suffices to show that 
\[
\De_n^{-1} R^n (u) \to \log \E\left[\exp(\im \lan u,  U_1\ran_{\text{tr}})\right]  \qquad \text{as } \De_n \to 0 
\]
to conclude $\varphi^n (u) \ucp \varphi (u) $ for each $u \in \R^{d\times d}$ and hence the convergence \eqref{Uconv}.
To compute the quantity $R^n (u)$ we recall that $M^n$ has the L\'evy triplet $(0,0,G(dx) \1_{\{\|x\| \leq v_n \}})$ and $A^n$ is a compound Poisson
process with intensity $\overline{G}(v_n):= G(\{x\in \R^d:~v_n<\|x\| \leq 1\})$ and jump distribution $G(dx) \1_{\{v_n<\|x\| \leq 1 \}}/ \overline{G}(v_n)$, and  $M^n$ and $A^n$ are independent. We also observe the identity 
$\lan u,  x y^{\top} \ran_{\text{tr}}= \lan x,  u y \ran = \lan y,  u^{\top} x \ran$ for any $x,y \in \R^d$ and $u\in \R^{d\times d}$. Hence, recalling the definition
of $\xi_i^n(1)$ we obtain the formula  
\bee
\E\big[ \exp(\im  \lan u , \zeta_1^n(1)\ran_{\text{tr}}) \big] = \E\big[ \exp(\im  \lan  \De_1^n M^n , \de_n(u+u^{\top}) \De L_{T(n,1)_1} \ran \1_{\{\tau(n,1)=1\}} )\big].
\eee
By conditioning we thus deduce that 
\begin{align*}
R^n(u) &= \alpha_n \De_n \int_{v_n<\|y\|\leq 1} \{\exp(z_n(u,y)) -1\} G(dy), \\
z_n(u,y) &=  \De_n \int_{\|x\| \leq v_n} \{\exp(\im \lan x, \de_n(u+u^{\top}) y \ran )- 1- \im \lan x, \de_n(u+u^{\top}) y \ran\} G(dx),
\end{align*} 
with $\alpha_n:=\exp(-\De_n \overline{G}(v_n)) \to 1$ since $\overline{G}(v_n) \leq C \De_n^{-1/2} (\log(1/\De_n))^{-\be}$. 
We now decompose $R^n(u)=\rho_n(u) + \rho'_n(u)$ with 
\begin{align*}
\rho_n(u) &= \al_n \De_n^2 \int_{v_n<\|y\|\leq 1} \left(\int_{\|x\| \leq v_n} \{\exp(\im \lan u, \de_n x\odot y  \ran_{\text{tr}} )- 1- \im \lan u, \de_n
x\odot y  \ran_{\text{tr}} G(dx) \right) G(dy), \\
 \rho'_n(u)& = \alpha_n \De_n \int_{v_n<\|y\|\leq 1} \{\exp(z_n(u,y)) -1 - z_n(u,y)\} G(dy).
\end{align*}
We observe that for any $w>0$ it holds that
\[
\int (w\|x\|) \wedge (w\|x\|)^2 G(dx)  = C\left(w \int_{w^{-1}}^1 r^{-\be} dr + w^2 \int_{0}^{w^{-1}} r^{1-\be} dr \right)  \le C w^{\beta}. 
\]
In conjunction with the inequality $|\exp(\im w)-1-\im w| \le C (|w| \wedge w^2 )$, we deduce, for a fixed $u \in \R^{d \times d}$, that 
\bee
|z_n(u,y)| \leq C\De_n \int (\de_n\|y\| \|x\|) \wedge (\de_n\|y\| \|x\|)^2 G(dx) \leq C \De_n \de_n^{\be} \|y\|^{\be} 
=C  \|y\|^{\be} / \log(1/\De_n).
\eee
Consequently, we obtain for a fixed $u \in \R^{d \times d}$
\bee
|\rho'_n(u)| \leq C \al_n \De_n \int_{v_n<\|y\|\leq 1} |z_n(u,y)|^2 G(dy) \leq C \al_n \De_n (\log(1/\De_n))^{-1} = o(\De_n).
\eee
This proves the approximation
\begin{align*}
\De_n^{-1} R^n(u) =  \De_n^{-1} \rho_n(u) + o(1)=  \alpha_n \int \{\exp(\im \lan u, z \ran_{\text{tr}}) - 1 - \im \lan u, z \ran_{\text{tr}}\}\nu_n (dz) +o(1),
\end{align*}
where $\nu_n$ is a L\'evy measure on $\R^{d\times d}$ defined by 
\bee \label{Fn}
\nu_n(A) =   \De_n \int_{v_n<\|y\|\leq 1} \left(\int_{\|x\| \leq v_n} \1_A(\de_n x\odot y) G(dx) \right) G(dy).
\eee
Note that $\nu_n$ is a L\'evy measure that satisfies $\int(\|x\|_{\text{tr}} \wedge \|x\|_{\text{tr}}^2) \nu_n(dx)<\infty$. 
To analyse the convergence of the measure $\nu_n$ we observe the following identities: For $x, y \in \R^d$ it holds that 
$\text{tr}(x\odot y)=2\lan x,y \ran$ and $\|x\odot y\|_{\text{tr}}^2=2(\|x\|^2 \|y\|^2 + \lan x,y \ran^2)$. 

In the final step we will show that the L\'evy measure $\nu_n$ converges, which requires to prove the conditions of \cite[Theorem 8.7]{S99}. Recall that $\S_{d\times d}$ denotes the unit sphere on $\R^{d \times d}$ equipped with the Frobenius norm $\|\cdot\|_{\text{tr}}$. 
Since $\int_{\|z \|_{\text{tr}} \geq 1 } \lan u, z \ran_{\text{tr}} \nu_n (dz) =0$ due to the symmetry of the measure $G$,
it suffices to show the
following conditions:
\begin{align}
\label{cond1a} 
&\lim_{n\to \infty} \nu_n(B\times (w,\infty)) = \nu_{U}(B\times (w,\infty)) \qquad \text{for } B\in \mathcal{B}(\S_{d\times d}) , w>0, \\[1.5 ex] 
\label{cond2a} 
&  \lim_{\ep \to 0} \limsup_{n \to \infty}   \int_{z:~\|z\|_{\text{tr}}\leq \ep} \|z\|_{\text{tr}}^2 \nu_n(dz) = 0.
\end{align}    
First of all, we observe that the support of $\nu_U$ must be contained in $\R^d \odot \R^d$. We start by showing the condition \eqref{cond1a}. 
Recalling the definition of the L\'evy measure $G$, we obtain the identity 
\begin{align*}
\nu_n(B\times (w,\infty)) &= \De_n \int_{\S_d^2} \int_{(0,v_n) \times (v_n, 1)} \1_B\left(\frac{\te_1\odot \te_2}{\|\te_1\odot \te_2\|_{\text{tr}}} \right)
\1_{(w,\infty)}\left(\de_n \rho_1 \rho_2 \sqrt{2(1+\lan \te_1,\te_2 \ran^2)}\right) \\[1.5 ex]
& \quad \times (\rho_1 \rho_2)^{-1-\be}d\rho_1 d\rho_2 H(d\te_1) H(d\te_2).
\end{align*}  
We compute the integral with respect to $d\rho_1 d\rho_2$. Observe that $\de_n v_n^2 \to \infty$ and hence, for any $r>0$, it holds that
\begin{align} \label{intcomp}
&\De_n \int_{(0,v_n) \times (v_n, 1)} \1_{(w,\infty)}\left(\de_n \rho_1 \rho_2r \right) (\rho_1 \rho_2)^{-1-\be} d\rho_1 d\rho_2 \\[1.5 ex] 
&= \De_n \int_{(v_n, 1)} \left( \int_{w/(\de_n \rho_2 r)}^{v_n} \rho_1^{-1-\be} d\rho_1 \right) \rho_2^{-1-\be} d\rho_2 \nonumber \\[1.5 ex] 
&= \frac{\De_n \de_n^{\be} r^{\be}}{\be w^{\be}} \int_{(v_n, 1)} \rho_2^{-1} d\rho_2 +o(1) = \frac{ r^{\be}}{2\be^2 w^{\be}} +o(1),  
\end{align}
where we used that $2\be \log(1/v_n)/ \log(1/\De_n) \to 1$. Hence, we conclude the convergence in \eqref{cond1a} with 
\bee
\nu_U(A) = \frac{1}{2\be} \int_{\S_{d\times d}} \mu(dz) \int_0^{\infty} \1_A(\rho z) \rho^{-1-\be} d\rho, \qquad A\in \mathcal{B}(\R^d \odot \R^d),
\eee 
and 
\bee
\mu(z) = \int_{\S_d^2} \1_z\left(\frac{\te_1\odot \te_2}{\|\te_1\odot \te_2\|_{\text{tr}}} \right) \left(2(1+\lan \te_1,\te_2 \ran^2)\right)^{\be/2} H(d\te_1) H(d\te_2), \qquad z\in \mathcal{B}(\S_{d\times d}).
\eee
Now, we turn our attention to condition \eqref{cond2a}. For $\ep>0$ we have that 
{\small\begin{align*}
\int_{z:~\|z\|_{\text{tr}}\leq \ep} \|z\|_{\text{tr}}^2 \nu_n(dz) &= \De_n \int_{\S_d^2} \int_{(0,v_n) \times (v_n, 1)} \1_{\S_{d\times d}}
\left(\frac{\te_1\odot \te_2}{\|\te_1\odot \te_2\|_{\text{tr}}} \right)
\1_{(0,\ep)}\left(\de_n \rho_1 \rho_2 \sqrt{2(1+\lan \te_1,\te_2 \ran^2)}\right) \\[1.5 ex]
&\quad \times (\rho_1 \rho_2)^{-1-\be} \left(\de_n \rho_1 \rho_2 \sqrt{2(1+\lan \te_1,\te_2 \ran^2)}\right)^2 d\rho_1 d\rho_2 H(d\te_1) H(d\te_2).
\end{align*}} 
A similar computation as in \eqref{intcomp} gives for any $r>0$:
\begin{align} \label{intcomp1}
&\De_n \de_n^2 r^2 \int_{(0,v_n) \times (v_n, 1)} \1_{(0,\ep)}\left(\de_n \rho_1 \rho_2r \right) (\rho_1 \rho_2)^{1-\be} d\rho_1 d\rho_2 \\[1.5 ex] 
&= \De_n \de_n^2 r^2 \int_{(v_n, 1)} \left( \int_{0}^{\ep/(\de_n \rho_2 r)} \rho_1^{1-\be} d\rho_1 \right) \rho_2^{1-\be} d\rho_2 \nonumber \\[1.5 ex] 
&\to  \frac{ r^{\be} \ep^{2-\be}}{2\be(2-\be)}  \qquad \text{as } n\to \infty.  \nonumber 
\end{align}
Since $\be \in (0,2)$ we conclude that 
\begin{align}
\lim_{\ep \to 0} \limsup_{n \to \infty}   \int_{z:~\|z\|_{\text{tr}}\leq \ep} \|z\|_{\text{tr}}^2 \nu_n(dz) = 0,
\end{align}
which proves the statement \eqref{cond2a} and completes the proof of Theorem \ref{th1} in the case $\beta>1$.

\subsection{Proof of Theorem \ref{th1} in the case $\be \in (0,1]$}

\subsubsection{Negligible terms}
Recall that $v_n= (\De_n \log(1/\De_n))^{1/\be}=\de_n^{-1}$ when $\be\leq 1$. In this subsection we show that the terms $\sum_{i=1}^{\lf t/\De_n \rf} \xi_i^n(j)$ are negligible for $j=1,3,4,5$. We start with $j=4$. If $\be=1$, 
we observe that $\E[\xi_i^n(4)| \f_{(i-1)\De_n}]=0$ and 
\[
\E[\|\xi_i^n(4)\|_{\text{op}}^2] \leq C \De_n^2 \de_n^2 v_n^{2} = C\De_n^2,
\] 
where the above inequality follows from \eqref{mnineq1} and \eqref{mnineq2}.  Hence, $\sum_{i=1}^{\lf t/\De_n \rf} \xi_i^n(4) \ucp 0$ by \eqref{cond2} when $\be=1$. For 
$\be<1$ we use the inequalities  \eqref{mnineq1} and \eqref{mnineq2} with $p=1$ to obtain that
\[
\E[\|\xi_i^n(4)\|_{\text{op}}] \leq  C \De_n^2 \de_n v_n^{2(1-\be)} = o(\De_n).
\]
Thus, we conclude that $\sum_{i=1}^{\lf t/\De_n \rf} \xi_i^n(4) \ucp 0$ by \eqref{cond1} when $\be<1$.

Now, we consider the case $j=3$. Notice that $\| x\odot y\|_{\text{op}} \leq 2\|x\| \|y\|$ for all $x,y\in \R^d$. Hence, we obtain that 
\[
\|\xi_i^n(3)\|_{\text{op}} \leq C \de_n \|\De_i^n M^n \| \sum_{j\geq 1} \|\De L_{T(n,i)_j} \| \1_{\{\tau(n,i) \geq 2 \vee j \}}.
\] 
Observing that $M^n$ and $A^n$ are independent, and applying \eqref{Zineq2} for $w=1$ along with \eqref{mnineq2} (either for $p=1$ when $\be<1$
or for $p=2$ when $\be=1$), we deduce the
inequality
\bee
\E\left[\|\xi_i^n(3)\|_{\text{op}} \right] \leq C
\begin{cases}
\de_n\De_n^{5/2} v_n^{-1/2}  \log(1/\De_n)  & \text{if } \be=1 \\
\de_n\De_n^{3}  v_n^{1-2\be}  & \text{if } \be<1.
\end{cases}
\eee
In both cases we have that $\sum_{i=1}^{\lf t/\De_n \rf}  \E\left[\|\xi_i^n(3)\|_{\text{op}} \right] \to 0$ and hence 
$\sum_{i=1}^{\lf t/\De_n \rf} \xi_i^n(3) \ucp 0$.

Next, we consider the case $j=1$. We start with $\be<1$. As in the previous case we get 
\[
\|\xi_i^n(1)\|_{\text{op}} \leq C \de_n \|\De_i^n M^n \|  \|\De L_{T(n,i)_1} \| \1_{\{\tau(n,i) =1\}}.
\] 
By inequalities \eqref{mnineq2} and \eqref{Zineq2} applied for $w=1$ and $m=j=1$, we then deduce that 
\bee
\E[\|\xi_i^n(1)\|_{\text{op}} ] \leq C \de_n \De_n^2 v_n^{1-\be} = \De_n (\log(1/\De_n))^{-1}.  
\eee 
This immediately implies that $\sum_{i=1}^{\lf t/\De_n \rf} \xi_i^n(1) \ucp 0$. When $\be=1$ we first observe that $\E[\xi_i^n(1)|~\f_{(i-1)\De_n}]=0$ due to independence of $A_n$ and $M_n$. On the other hand, using again  \eqref{mnineq2} and \eqref{Zineq2}, we obtain that
\bee
\E[\|\xi_i^n(1)\|_{\text{op}}^2|~\f_{(i-1)\De_n}] \leq C \de_n^2 \E[\|\De_i^n M^n \|^2  \|\De L_{T(n,i)_1} \|^2 \1_{\{\tau(n,i) =1\}}] \leq 
C \de_n^2 \De_n^2 v_n.
\eee
Hence, by \eqref{cond2} we conclude that  $\sum_{i=1}^{\lf t/\De_n \rf} \xi_i^n(1) \ucp 0$.

Finally, let us treat the case $j=5$. We use the decomposition $\xi_i^n(1)= \eta_i^n(1,2) + \eta_i^n(1,3) + \eta_i^n(2,3) +\eta_i^{\prime n}$ with 
\begin{align*}
\eta_i^n(j,k) &= \de_n \De L_{T(n,i)_j} \odot \De L_{T(n,i)_k} \1_{\{\tau(n,i) =3\}}, \\
\eta_i^{\prime n} &= \de_n \sum_{r=2}^{\infty} \sum_{k=1}^{r-1} \De L_{T(n,i)_r} \odot \De L_{T(n,i)_k} \1_{\{\tau(n,i) \geq r \vee 4\}}.
\end{align*}
Recalling the inequality $(\sum_j a_j)^{\be} \leq \sum_j a_j^{\be}$ for positive real numbers $a_j$ and $\be\leq 1$, and applying  \eqref{Zineq3}
for $w=1$ we deduce that 
\begin{align*}
\E[\|\eta_i^{\prime n}\|_{\text{op}} \wedge 1 ] &\leq  \E[\|\eta_i^{\prime n}\|_{\text{op}} ^{\be} ] \leq C \de_n^{\be} 
\sum_{r=2}^{\infty} \sum_{k=1}^{r-1} \E\left[ \|\De L_{T(n,i)_r}\|^{\be} \|\De L_{T(n,i)_k}\|^{\be} \1_{\{\tau(n,i) \geq r \vee 4\}} \right] \\
& \leq C \de_n^{\be} \De_n^2 (\log(1/\De_n))^2 \sum_{r=2}^{\infty} \sum_{k=1}^{r-1} \left(\frac{b\De_n}{v_n^{\be}} \right)^{r \vee 4 -2}
 \le C \de_n^{\be} \De_n^4  (\log(1/\De_n))^2 / v_n^{2\be} \\
&= C \De_n/ \log (1/\De_n) .
\end{align*} 
Hence, via \eqref{cond1} we conclude that $\sum_{i=1}^{\lf t/\De_n \rf} \eta_i^{\prime n} \ucp 0$. When $\be<1$ we deduce from \eqref{Zineq3}
applied to $w=1/\de_n$ that 
\[
\E[\|\eta_i^n(j,k)\|_{\text{op}} \wedge 1 ] \leq C \de_n^{\beta} \De_n^3 \log(1/\De_n)/v_n^{\be} = C \De_n / \log (1/\De_n), 
\qquad 1\leq j<k\leq 3,
\]
and thus we again conclude that $\sum_{i=1}^{\lf t/\De_n \rf} \eta_i^n(j,k) \ucp 0$. In the setting $\be=1$ we will show that condition 
\eqref{cond3} is satisfied. First, we observe that random variables $\eta_i^n(j,k)$ have symmetric distribution since $G$ is symmetric. 
Consequently, we deduce the identity 
\[
\E\left[\eta_i^n(j,k) \1_{\{ \|\eta_i^n(j,k)\|_{\text{\rm op}}\leq 1  \}}|~\f_{(i-1)\De_n}\right] =0.
\] 
On the other hand, using again \eqref{Zineq3}
for  $w=1/\de_n$  and $p=2$ we obtain that 
\[
\E\left[(\|\eta_i^n(j,k)\|_{\text{op}} \wedge 1)^2|~\f_{(i-1)\De_n}\right] \leq C \de_n \De_n^3 \log(1/\De_n)/v_n =C \De_n / \log (1/\De_n).
\]
Thus, all conditions of  \eqref{cond3} are satisfied and we conclude that $\sum_{i=1}^{\lf t/\De_n \rf} \eta_i^n(j,k) \ucp 0$ for $1\leq j<k\leq 3$
and $\be=1$.

\subsubsection{The dominating term}

In this subsection we prove that $\sum_{i=1}^{\lf t/\De_n \rf} \xi_i^n(2) \schw U_t$. 
We will apply \cite[Theorem VII.3.4]{JS03}, which holds for partial sums of independent random variables. The following  conditions are sufficient
to guarantee the functional weak convergence $\sum_{i=1}^{\lf t/\De_n \rf} \xi_i^n(2) \schw U_t$:
\begin{align} 
\label{cltcond1} &\sum_{i=1}^{\lf t/\De_n \rf} \E\left[ \xi_i^n(2) \1_{\{ \| \xi_i^n(2)\|_{\text{tr}} \leq 1 \}} \right] \ucp 0, \\
\label{cltcond2} & \lim_{\ep \to 0} \limsup_{n \to \infty} \sum_{i=1}^{\lf t/\De_n \rf} \E\left[ \| \xi_i^n(2)\|_{\text{tr}} ^2 \1_{\{ \| \xi_i^n(2)\|_{\text{tr}} \leq \ep \}}
\right]
=0 \qquad \text{for any } t>0, \\[1.5 ex]
\label{cltcond3} & \lim_{n\to \infty} \sum_{i=1}^{\lf t/\De_n \rf} \P(\xi_i^n(2) \in A) = t \nu_U(A),
\end{align}
where the last condition holds for all sets of the form $A=\{\rho B:~ B\in \mathcal B(\S_{d} \odot \S_d), ~\rho \in (w,\infty)\}$ with $w>0$. In our setting
conditions \eqref{cltcond1}-\eqref{cltcond3} simplify even further, because $\xi_i^n(2)$ are identically distributed.

Since the L\'evy measure $G$ is
symmetric, we immediately deduce that the expectation $\E[ \xi_i^n(2) \1_{\{ \| \xi_i^n(2)\|_{\text{tr}} \leq 1 \}} ] =0$ and hence  \eqref{cltcond1} holds. We proceed with the proof of condition  \eqref{cltcond2}. For $x, y \in \R^d$ it holds $\|x\odot y\|_{\text{tr}}^2=2(\|x\|^2 \|y\|^2 + \lan x,y \ran^2)$ from which we deduce that
\bee\label{eq:tool}
\sqrt{2} \|x\| \|y\| \le \| x\odot y \|_{\text{tr}}\le 2 \|x\| \|y\|\,.
\eee
%To this end, we compute
%\begin{align*}
%\| x\odot y \|_{\text{tr}}&\le \| xy^{\top} \|_{\text{tr}}+\| yx^{\top} \|_{\text{tr}}
%= \big(\text{tr}(yx^{\top}xy^{\top})\big)^{1/2}+\big(\text{tr}(xy^{\top}yx^{\top})\big)^{1/2}= 2\|x\| \|y\|,\\
%\| x\odot y \|_{\text{tr}}^2&= \sum_{i,j=1}^d (x_iy_j+y_ix_j)^2=2 \|x\|^2 \|y\|^2+2 \lan x,y \ran^2
%\end{align*} 
Recalling the definition of the term $\xi_i^n(2)$ and %observing the inequality $\| x\odot y \|_{\text{tr}}\geq \sqrt{2} \|x\| \|y\|$ for any $x,y \in \R^d$
using \eqref{eq:tool}, we get the inequality 
\begin{align*}
\E\left[ \| \xi_i^n(2)\|_{\text{tr}} ^2 \1_{\{ \| \xi_i^n(2)\|_{\text{tr}} \leq \ep \}}
\right] &\leq C \de_n^2 \E\left[ \| \De L_{T(n,i)_1}\| ^2 \| \De L_{T(n,i)_2}\| ^2 \right. \\
& \left. \times  \1_{\{\tau(n,i)=2,~ \de_n \| \De L_{T(n,i)_1}\|\| \De L_{T(n,i)_2}\| \leq \ep/\sqrt{2} \}}
\right] =: r_i^n.
\end{align*}
Noting that $\de_n v_n=1$, we conclude that
\begin{align*}
r_i^n &\leq C \De_n^2 \de_n^2 \int_{(v_n,1]^2} (\rho_1 \rho_2)^{1-\be}  \1_{\{ \de_n \rho_1 \rho_2 \leq \ep/\sqrt{2}  \}} d\rho_1 d\rho_2 \\
&= C \De_n^2 \de_n^2 \int_{v_n}^{1} \rho_1^{1-\be} \left( \int_{v_n}^{\ep/(\sqrt{2} \rho_1 \de_n)} \rho_2^{1-\be} d\rho_2 \right) d\rho_1 \\
& \leq \De_n^2 \de_n^{2} \left( C\frac{\ep^{2-\be} \de_n^{\be -2}}{\log \de_n}+C v_n^{2-\be} \right)
\leq C \De_n^2 \de_n^{\be}\left( \frac{ \ep^{2-\be} }{\log \de_n}+1\right).
\end{align*}
The latter implies
$\limsup_{n \to \infty} \De_n^{-1} r_i^n =0$, 
and consequently condition \eqref{cltcond2} holds.

Now, we show the condition \eqref{cltcond3} for the announced sets $A$. First of all, we deduce that 
\bee
\De_n^{-1} \P(\xi_i^n(2) \in A) = \frac{\De_n K}{2} \int_{(v_n,1]^2} (\rho_1 \rho_2)^{-1-\be}  \1_{(w,\infty)}(\de_n \rho_1 \rho_2) d\rho_1 d\rho_2
+ o(1), 
\eee
where $K=\int_{\S_d^2} \1_{B} (x\odot y) H(dx) H(dy)$.   Assume for the moment that $w\in (0,1)$. In this case we have that $w/(\de_n \rho_1)<1$ for any $\rho_1 \geq v_n$. Hence, we get that 
\begin{align*}
&\frac{\De_n K}{2} \int_{(v_n,1]^2} (\rho_1 \rho_2)^{-1-\be}  \1_{(w,\infty)}(\de_n \rho_1 \rho_2) d\rho_1 d\rho_2 \\
&= \frac{\De_n K}{2} \left( \int_{v_n}^{w} \rho_1^{-1-\be} \left( \int_{w/(\de_n \rho_1)}^1 \rho_1^{-1-\be} d\rho_2 \right) d\rho_1 +
\frac{\de_n^{\be}}{\be} \int_w^1  \rho_1^{-1-\be} d \rho_1
\right) \\
& = \frac{K \De_n \de_n^{\be} \log(\de_n)}{2\be w^{\be}} + o(1) = \frac{K}{2\be^2 w^{\be}} + o(1).
\end{align*}
Consequently, we obtain the convergence in  \eqref{cltcond3} for $w\in (0,1)$. For $w\geq 1$ we have that $w/(\de_n \rho_1)\le 1$ for any 
$\rho_1 \geq w v_n$. Thus, we deduce that 
\begin{align*}
&\frac{\De_n K}{2} \int_{(v_n,1]^2} (\rho_1 \rho_2)^{-1-\be}  \1_{(w,\infty)}(\de_n \rho_1 \rho_2) d\rho_1 d\rho_2 \\
&= \frac{\De_n K}{2} \int_{wv_n}^{1} \rho_1^{-1-\be} \left( \int_{w/(\de_n \rho_1)}^1 \rho_1^{-1-\be} d\rho_2 \right) d\rho_1 
 \\
&= \frac{\De_n K}{2 \be} \int_{wv_n}^{1} \rho_1^{-1-\be}  \left((w/(\de_n \rho_1))^{-\be} -1 \right) d\rho_1  \\
& = \frac{K \De_n \de_n^{\be} \log(\de_n)}{2\be w^{\be}} + o(1) = \frac{K}{2\be^2 w^{\be}} + o(1).
\end{align*}
This implies that the convergence in  \eqref{cltcond3} also holds for $w\geq 1$, which completes the proof.

\bibliographystyle{chicago}
 
\end{document}